\documentclass[reqno]{amsart}
\usepackage{amsmath, amssymb, amsfonts}
\usepackage{xcolor}
\newtheorem{theorem}{Theorem}

\newtheorem{corollary}{Corollary}
\newtheorem{lemma}{Lemma}
\newtheorem{proposition}{Proposition}
\theoremstyle{remark}
\newtheorem{remark}{Remark}
\theoremstyle{definition}

\newtheorem*{Acknowledgment}{Acknowledgment}

\newcommand\supp{\mathop{\rm supp}}

\begin{document}

\title[A note on pseudoconvex hypersurfaces of infinite type]{A note on pseudoconvex hypersurfaces of infinite type in $\mathbb C^n$}

\author{John Erik Forn{\ae}ss and Ninh Van Thu\textit{$^{1,2}$}}

\thanks{The second author was supported by the Vietnam National Foundation for Science and Technology Development (NAFOSTED) under Grant Number 101.02-2017.311.}

\address{John Erik Forn{\ae}ss}
\address{Department of Mathematics, NTNU, Sentralbygg 2, Alfred Getz vei 1, 7491 Trondheim, Norway}
\email{john.fornass@ntnu.no}
\address{Ninh Van Thu}
\address{\textit{$^{1}$}~Department of Mathematics,   VNU University of Science, Vietnam National University at Hanoi, 334 Nguyen Trai, Thanh Xuan, Hanoi, Vietnam}
 \address{\textit{$^{2}$}~Thang Long Institute of Mathematics and Applied Sciences,
Nghiem Xuan Yem, Hoang Mai, HaNoi, Vietnam}
\email{thunv@vnu.edu.vn}

\begin{abstract}
The purpose of this article is to prove that there exists a real smooth  pseudoconvex hypersurface germ $(M,p)$ of  D'Angelo infinite type in $\mathbb C^{n+1}$ such that it does not admit any (singular) holomorphic curve in $\mathbb C^{n+1}$ tangent to $M$ at $p$ to infinite order. 
\end{abstract}
\maketitle
\allowdisplaybreaks

\section{Introduction}
Let $(M, p)$ be a smooth real hypersurface germ at $p$ in $\mathbb C^{n+1}$ and let $r$ be a local defining function for $M$ near $p$. Suppose that $(M, p)$ is of D'Angelo infinite type, i.e., there exists a sequence of nonconstant holomorphic curves $\gamma_m: (\mathbb C,0)\to (\mathbb C^{n+1}, p)$ such that $\dfrac{\nu (r\circ \gamma_m)}{\nu(\gamma_m)}\to +\infty$ as $m\to \infty$, where $\nu(f)$ denotes the vanishing order of $f$ at $0$. It is natural to ask whether there exists a nonconstant holomorphic curve $\gamma: (\mathbb C,0)\to (\mathbb C^{n+1}, p)$ tangent to $M$ at $p$ to infinite order, i.e. $\nu (r\circ \gamma)=+\infty$. 

This question plays a crucial role in the regularity of $\bar \partial$-Neumann problems over pseudoconvex domains (see \cite{Da82, Cat83, Cat84, Cat87, DK99}, and the references therein). The main results around this question are due to T. Bloom and I. Graham \cite{BG77}, L. Lempert and J. P. D'Angelo \cite{Da93, Lem86}, the first author and B. Stens{\o}nes \cite{FS12}, the first author, L. Lee and Y. Zhang \cite{FLZ14}, and K.-T. Kim and the second author \cite{Kim-Ninh}.
 
 If $(M, p)$ is real-analytic, it was shown that $M$ contains a nontrivial holomorphic curve $\gamma_\infty$ passing through $p$ (see  \cite{Da93, Lem86, FS12}). For the case when $(M,p)$ is a smooth real hypersurface in $\mathbb C^{n+1}$, the first author, L. Lee and Y. Zhang \cite{FLZ14} proved that there exists a formal complex curve in the hypersurface $M$ through $p$. Recently, K.-T. Kim and the second author \cite{Kim-Ninh} and T. K. S. Nguyen and V. T. Chu \cite{NC19} showed that in general there is no such a holomorphic curve. However, the hypersurfaces constructed in \cite{Kim-Ninh} and \cite{NC19} are not pseudoconvex.
 
 In this paper, we ensure that this result still holds even for higher-dimensional pseudoconvex hypersurfaces and for singular holomorphic curves. More precisely, we prove the following theorem.
 \begin{theorem}\label{T1} Let $n\geq 1$. There exists a smooth pseudoconvex real hypersurface germ $(M,0)$ of D'Angelo infinite type in $\mathbb C^{n+1}$ that does not admit any nonconstant holomorphic curve $\gamma : (\mathbb C,0)\to (\mathbb C^{n+1},0)$ tangent to $M$ at $0$ to infinite order.
 \end{theorem}
Theorem \ref{T1} is a crucial consequence of the main result of this paper. In order to state our main result, let us recall the notion of points of Bloom-Graham type.  A point $p\in M \subset \mathbb C^{n+1}$ is of Bloom-Graham type $m$ ($m$ is a positive integer or $+\infty$) if $m$ is the supremum of the orders of tangency of $M$ and codimension one complex submanifolds of $\mathbb C^{n+1}$ at $p$. It was proved in \cite[Theorem $2.4$]{BG77} that a point $p\in M$ is of Bloom-Graham type $m$ if and only if $p$ is of type $m$ in the sense of J. J. Kohn, defined in terms of iterated commutators of vector fields. We remark here that for smooth real hypersurfaces in $\mathbb C^2$, D'Angelo finite type, Bloom-Graham finite type, and Kohn finite type are equivalent. For various notions of points of finite type and their relationships to subelliptic estimates, we refer the reader to \cite{DK99} and the references therein. 

Let $(M, p)$ be a smooth real hypersurface germ at $p$ in $\mathbb C^{n+1}$. Suppose that $(M, p)$ is of Bloom-Graham infinite type. Then, $(M, p)$ is also of D'Angelo infinite type (see Lemma \ref{BL-1} in Section \ref{S3}). As for the above mentioned notion of D'Angelo infinite type, it is natural to ask whether there exists a nonconstant holomorphic curve $\gamma: (\mathbb C,0)\to (\mathbb C^{n+1}, p)$ tangent to $M$ at $p$ to infinite order if $(M,p)$ is of Bloom-Graham infinite type. In \cite[Counterexamples $2.14$]{BG77}, T. Bloom and I. Graham introduced a smooth real hypersurface germ $(M,0)$ of  infinite type in $\mathbb C^2$ that does not admit any complex submanifold tangent to $M$ at the origin to infinite order. 

Following the argument given in \cite[Counterexamples $2.14$]{BG77}, one sees that the hypersurface constructed in \cite{BG77} may not be pseudoconvex. In this paper, we show that there exists a smooth pseudoconvex real hypersurface germ $(M, p)$ of Bloom-Graham infinite type in $\mathbb C^{n+1}$ such that there is no nonconstant holomorphic curve in $\mathbb C^{n+1}$ tangent to $M$ at $p$ to infinite order. Namely, we prove the following theorem as our main result.
 \begin{theorem}\label{T4} Let $n\geq 1$. There exists a smooth pseudoconvex real hypersurface germ $(M,0)$ of Bloom-Graham infinite type in $\mathbb C^{n+1}$ that does not admit any nonconstant holomorphic curve in $\mathbb C^{n+1}$ tangent to $M$ at $0$ to infinite order.
 \end{theorem}
The proof of  Theorem \ref{T4} is split into several steps. First, for an increasing sequence of positive real numbers  $\{a_m\}_{m=1}^\infty $, we construct a $\mathcal{C}^\infty$-smooth subharmonic function $f$ on $\mathbb C$ such that its Taylor series at the origin is exactly $\sum\limits_{m=1}^\infty  {\mbox{Re}} (a_m z^m)$ (cf. Proposition \ref{P1} in Section \ref{S2}). Next, choose $n$ suitable sequences of positive real numbers $\{a^1_m\}_{m=1}^\infty, \ldots,\{a^n_m\}_{m=1}^\infty$ and let $f_1,\ldots, f_n$ be $\mathcal{C}^\infty$-smooth subharmonic functions $\mathbb C$ constructed as in Proposition \ref{P1} with respect to these sequences. Then, the desired hypersurface $M$ is defined by
  $$
  M=\left\{(z,w)\in \mathbb C^{n+1}\colon {\mbox{Re}} (w)+f_1(z_1)+\cdots+f_n(z_{n})=0\right\}.
  $$
 
 As a consequence of Theorem \ref{T4}, we obtain the following corollary. 
 \begin{corollary}\label{Cor1} Let $n\geq 1$. There exists a smooth pseudoconvex real hypersurface germ $(M,0)$ of Bloom-Graham infinite type in $\mathbb C^{n+1}$ that does not admit any $n$-dimensional complex submanifold tangent to $M$ at $0$ to infinite order.
 \end{corollary}
\section{Construction of a $\mathcal{C}^\infty$-smooth subharmonic function}\label{S2}
In this section, we shall prove the following proposition.
 \begin{proposition} \label{P1}
 Let $\{a_m\}_{m=1}^\infty $ be an increasing sequence of positive real numbers. Then, there exists a $\mathcal{C}^\infty$-smooth subharmonic function $f$ on $\mathbb C$ satisfying that its Taylor series at the origin is exactly $\sum\limits_{m=1}^\infty  {\mbox{Re}} (a_m z^m)$.
 \end{proposition}

In order to give a proof of Proposition \ref{P1}, we need following lemmas. First of all, denote by $\chi$ a nonnegative $\mathcal{C}^\infty$-smooth cut-off function on $\mathbb R$ such that
\[  
  \chi(t)=\begin{cases}
1 ~&~\text{if}~t<1/4\\
0~&~\text{if}~t>1.
  \end{cases}
  \]
Let $\{a_m\}_{m=1}^\infty $be a given increasing sequence of positive real numbers. Denote by $\{\epsilon_m\}_{m=1}^\infty $ an increasing sequence of positive real numbers such that $\epsilon_m\geq\max\{m,a_m^{2/m}\}$ for every $m=1,2,\ldots$. Then, for each $m=1,2,\ldots$, denote $u_m(z)$ by setting
$$
u_m(z):=  \chi\left(\epsilon_m^2 |z|^2\right){\mbox{Re}}\left (a_mz^m \right).
$$
Then, we have 
\begin{equation*}
\begin{split}
\frac{\partial^2 u_m}{\partial z\partial \bar z}(z)&=   \epsilon_m^2 \chi'\left(\epsilon_m^2|z|^2\right){\mbox{Re}} \left (a_mz^m \right)+\epsilon_m^4 |z|^2 \chi''\left(\epsilon_m^2|z|^2\right){\mbox{Re}}\left (a_mz^m \right)\\
&+ m\epsilon_m^2\chi'\left(\epsilon_m^2|z|^2\right) {\mbox{Re}} \left (a_mz^m \right)\\
&=\epsilon_m^2(m+1)\chi'\left(\epsilon_m^2|z|^2\right){\mbox{Re}}\left (a_mz^m \right)+  \epsilon_m^4 |z|^2 \chi''\left(\epsilon_m^2|z|^2\right){\mbox{Re}} \left (a_mz^m \right)
\end{split}
\end{equation*}
for all $z\in \mathbb C^*$. Thus, one has the following lemma.
\begin{lemma} \label{L1} For each $m=1,2,\ldots$, the following assertions hold:
\begin{itemize}
\item[i)] $\Delta u_m(z)=0$ for all $z\in \mathbb C$ with $|z|\leq \dfrac{1}{2\epsilon_m}$ or $|z|\geq \dfrac{1}{\epsilon_m}$,
\item[ii)] $\left |\Delta u_m(z)\right |\lesssim \dfrac{m a_m\epsilon_m^2}{\epsilon_m^m}$ for all $z\in \mathbb C$ with $\dfrac{1}{2\epsilon_m}<|z|<\dfrac{1}{\epsilon_m}$, where the constant is independent of $m$.
\end{itemize}
\end{lemma}
Next, let us denote by $\Lambda$ a $\mathcal{C}^\infty$-smooth convex function on $\mathbb R$ such that 
\begin{itemize}
\item[a)] $\Lambda(x)=0$ if $x\leq -2$,
\item[b)] $\Lambda''(x)>0$ if  $-2<x< 2$,
\item[c)] $\Lambda'(x)$ is constant if  $x\geq 2$.
\end{itemize}
Define $v_m(z):=C\,\dfrac{ma_m}{\epsilon_m^m} \Lambda\left(\log |z|^2+2\log \epsilon_m\right)$ for every $m=1,2,\ldots$, where $C>0$ will be chosen later. Then, the function $v_m$ is subharmonic on $\mathbb C$ for every $m=1,2,\ldots$. Moreover, we obtain the following lemma.
\begin{lemma} \label{L2} There exists a positive constant $C'>0$ such that
$$
\Delta v_m(z)\geq C' \dfrac{m a_m\epsilon_m^2}{\epsilon_m^m}
$$
for every $m=1,2,\ldots$ and for all $z\in \mathbb C$ with  $\dfrac{1}{2\epsilon_m}<|z|<\dfrac{1}{\epsilon_m}$.
\end{lemma}
\begin{proof}
A direct computation shows that
\begin{align*}
\frac{\partial^2 v_m}{\partial z\partial \bar z}(z)&=C\,\dfrac{m a_m}{\epsilon_m^m} \Lambda''\left(\log |z|^2+2\log \epsilon_m\right) \left | \frac{\partial}{\partial z}\left(\log |z|^2 \right)\right |^2\\
&+C\,\dfrac{m a_m}{\epsilon_m^m} \Lambda'\left(\log |z|^2+2\log \epsilon_m\right) \frac{\partial^2}{\partial z\partial \bar z}\left(\log |z|^2\right)\\
&=C\,\dfrac{m a_m}{\epsilon_m^m} \Lambda''\left (\log |z|^2+2\log \epsilon_m\right) \frac{1}{|z|^2}\\
&\geq C\,\dfrac{m a_m\epsilon_m^2}{\epsilon_m^m}\Lambda''\left(\log |z|^2+2\log \epsilon_m\right) \\
&\gtrsim \dfrac{m a_m\epsilon_m^2}{\epsilon_m^m}
\end{align*}
for all $z\in \mathbb C$ with $\dfrac{1}{2\epsilon_m}<|z|<\dfrac{1}{\epsilon_m}$, where the positive constant is independent of $m$. The proof is complete.
\end{proof}
It follows from Lemmas \ref{L1} and \ref{L2} that if $C$ is chosen fixed and large enough, then $u_m+v_m$ are all subharmonic. Furthermore, we have the following lemmas.
\begin{lemma} \label{L3}
If $|z|<\dfrac{1}{e\epsilon_m}$, then $u_m(z)+v_m(z)={\mbox{Re}}\left(a_m z^m\right)$ for every $m=1,2,\ldots$.
\end{lemma}
\begin{proof}
Fix a positive integer $m$. Then, $u_m(z)={\mbox{Re}}\left(a_m z^m\right)$ and $v_m(z)=0$ for all $z\in \mathbb C$ with  $|z|<\dfrac{1}{e \epsilon_m}$. Therefore, the proof follows.
\end{proof}
\begin{lemma} \label{L4}
The sums $\sum\limits_{m=1}^\infty u_m+v_m$ are uniformly convergent on compact sets in any $\mathcal C^k$ norm.
\end{lemma}
\begin{proof} 
Let $K$ be a fixed compact subset in $\mathbb C$ and fix a nonnegative integer $k$. Since $\supp (u_m)\subset \{z\in \mathbb C\colon  |z|\leq1/\epsilon_m\}$ for every $m=1,2,\ldots$, a computation shows that
\begin{align*}
\sup_{z\in K}\left |\frac{\partial^k u_m(z)}{\partial z^j \partial\bar z^{k-j}}\right | &=\sup_{  |z| \leq 1/\epsilon_m}\left |\frac{\partial^k u_m(z)}{\partial z^j \partial \bar z^{k-j}}\right |\\
&\leq  C_k  \dfrac{ a_m\epsilon_m^{2k}}{\epsilon_m^m}
\end{align*}
for all $m\geq k$ and $0\leq j\leq k$, where $C_k$ is a positive constant depending on $k$. Notice that $a_m\leq \epsilon_m^{m/2}$ and $\epsilon_m \geq m$ for every $m=1,2,\ldots$. Therefore, by Weierstrass $M$-test the following series $\sum\limits_{m=1}^\infty \frac{\partial^k u_m(z)}{\partial z^j \partial\bar z^{k-j}}$ are uniformly convergent on any compact subsets of $\mathbb C$ for any nonnegative integers $k,j$.

On the other hand,  since $\Lambda(x)\leq C_1( |x|+1) $ for all $x\in \mathbb R$, where $C_1>0$ is a constant and $\supp (v_m)\subset \{z\in \mathbb C\colon  |z|\geq \dfrac{1}{e \epsilon_n}\}$, it follows that
\begin{align*}
\sup_{z\in K}\left |\frac{\partial^k v_m(z)}{\partial z^j \partial\bar z^{k-j}}\right | &=\sup_{  |z| \geq 1/(e \epsilon_m ) }\left |\frac{\partial^k v_m(z)}{\partial z^j \partial \bar z^{k-j}}\right |\\
&\leq  {\widetilde C}_k \dfrac{m a_m\epsilon_m^k}{\epsilon_m^m}
\end{align*}
for all $m\geq k$ and $0\leq j\leq k$, where ${\widetilde C}_k$ is a positive constant depending on $k$ and $K$.  Note that $a_m\leq \epsilon_m^{m/2}$ and $\epsilon_m \geq m$ for every $m=1,2,\ldots$. Hence, by Weierstrass $M$-test the following series $\sum\limits_{m=1}^\infty \frac{\partial^k v_m(z)}{\partial z^j \partial\bar z^{k-j}}$ are also uniformly convergent on any compact subsets of $\mathbb C$ for any nonnegative integers $k,j$.
 
 Altogether, the proof is now complete.
\end{proof}
The following corollary immediately follows from Lemma \ref{L4}.
\begin{corollary} \label{L5}
$\sum\limits_{m=1}^\infty u_m+v_m$ is a $\mathcal C^\infty$ function and any derivative at $0$ is the sum of the corresponding derivatives of the $u_m+v_m$. 
\end{corollary}
\begin{proof}[Proof of Proposition \ref{P1}]
Define $f(z)=\sum\limits_{m=1}^\infty u_m(z)+v_m(z)$. Then, $f$ is a $\mathcal{C}^\infty$-smooth subharmonic function on $\mathbb C$. Moreover, by Lemma \ref{L3} and Corollary \ref{L5}, the Taylor series of $f$ at the origin is exactly $\sum\limits_{m=1}^\infty {\mbox{Re}}(a_m z^m)$. 
\end{proof}
\section{Pseudoconvex hypersurface of Bloom-Graham infinite type}\label{S3}
In this section, we shall give proofs of Theorem \ref{T4}, Theorem \ref{T1}, and Corollary \ref{Cor1}.
\begin{proof}[Proof of Theorem \ref{T4}] 
Fix an increasing sequence $\{k_s\}_{s=1}^\infty$ such that $k_{s+1}-k_s>n$ for every $s=1,2,\ldots$. Let $\{a^1_m\}_{m=1}^\infty, \ldots,\{a^{n-1}_m\}_{m=1}^\infty,\{a^0_m=a^n_m\}_{m=1}^\infty$ be $n$ increasing sequences of positive real numbers satisfying:
\begin{align} \label{5-10-2019-1}
a_{k_1}^1> (k_1)^{k_1}\Big((k_1)^{2k_1} \big(\sum_{m=1}^{k_1-1} a^1_m+\sum_{j\ne 1}\sum_{m=1}^{k_1} a^j_m\big)+1\Big)&\nonumber;\\
a_{k_2}^2> (k_2)^{k_2}\Big((k_2)^{2k_2} \big(\sum_{m=1}^{k_2-1} a^2_m+\sum_{j\ne 2}\sum_{m=1}^{k_2} a^j_m\big)+1\Big)&\nonumber;\\
\cdots\nonumber  \\
a_{k_s}^{s(\mathrm{mod}\ n)}> (k_s)^{k_s} \Big( (k_s)^{2 k_s} \big(\sum_{m=1}^{k_s-1} a^{s(\mathrm{mod}\ n)}_m+\sum_{j\ne s(\mathrm{mod}\ n)}\;\sum_{m=1}^{k_s} a^j_m\big)+1\Big)&;\\
\cdots \nonumber
\end{align}
Denote by $\{\epsilon^j_m\}_{m=1}^\infty, 1\leq j\leq n $, $n$ sequences of positive real numbers satisfying $\epsilon^j_m\geq \max\{m,(a^j_m)^{2/m}\}, 1\leq j\leq n$.

For each $j=1,\ldots,n$, let $f_j$ be a $\mathcal{C}^\infty$-smooth subharmonic function on $\mathbb C$ constructed in the proof of Proposition \ref{P1} for a pair of sequences $\{a^j_m\}_{m=1}^\infty$ and $\{\epsilon^j_m\}_{m=1}^\infty$. That is, the Taylor series at the origin of $f_j$ is $\sum\limits_{m=1}^\infty  {\mbox{Re}} \left(a^j_m z^m\right)$, $j=1,2,\ldots, n$.

We now define a hypersurface germ $M$ at $p=0$ by setting
\[
M=\left\{(z,w)\in\mathbb C^{n+1}\colon\tilde  \rho(z,w):={\mbox{Re}}(w)+f_1(z_1)+\cdots+f_n(z_n)=0\right\}.
\]
Since $f_k,1\leq k\leq n$, are subharmonic on $\mathbb C$, $M$ is pseudoconvex. Moreover, $(M,0)$ is of infinite type in the sense of Bloom-Graham.

 Indeed, for each $m=1,2,\ldots$, consider an $n$-dimensional complex submanifold $X_m$ in $\mathbb C^{n+1}$ defined by
 $$
 X_m=\left\{(z,w)\in \mathbb C^{n+1}\colon w=-\sum_{k=1}^m a^1_k z_1^k-\cdots-\sum_{k=1}^m a^n_k z_n^k , |z_j|<\frac{1}{e \epsilon^j_m },1\leq j\leq n\right\}.
 $$
 Then $\rho\mid_{X_m}(z,w)=o(|z_1|^m)+\cdots+o(|z_n|^m)$ vanishes to order $\geq m$ at $p=0$. Consequently, $X_m$ is tangent to $M$ at $0$ to order $\geq m$. This yields that $M$ is of infinite type in the sense of Bloom-Graham at $p=0$.

We now prove that there does not exist a nonconstant holomorphic curve $\gamma_\infty:=(h_1,\ldots, h_n; g ): (\mathbb C,0)\to (\mathbb C^{n+1},0)$, where $g, h_j, 1\leq j\leq n$, are holomorphic functions on a neighborhood of the origin in $\mathbb C$, such that $\nu(\rho\circ\gamma_\infty)=+\infty$, that is, 
\begin{align}\label{4-2018-1}
\rho\circ\gamma_\infty(t)={\mbox{Re}}(g(t))+f_1(h_1(t))+\cdots+f_n(h_n(t)))=o(t^\infty).
\end{align}
 Suppose otherwise that there exists such a holomorphic curve. Without loss of generality, we may assume that $g, h_j, 1\leq j\leq n$, are all holomorphic on the unit disk $\Delta:=\{z\in \mathbb C\colon |z|<1\}$.  
 
 We now consider the following cases
 
 \noindent
 {\bf Case 1.}  $h_j\equiv 0$ for  $1\leq j\leq n$. In this case we have $f_j(h_j(t))\equiv 0$ for $1\leq j\leq n$. Therefore, it follows from (\ref{4-2018-1}) that 
 $$
 {\mbox{Re}}(g(t))=o(t^\infty).
 $$
This implies that $g\equiv 0$, which is impossible. 

 \noindent
 {\bf Case 2.}  $g\equiv 0$. Then, (\ref{4-2018-1}) becomes
 \begin{align} \label{4-2018-20-1}
f_1(h_1(t))+\cdots+f_n(h_n(t)))=o(t^\infty).
\end{align}
 Expanding $h_1,\ldots, h_n$ into the Taylor series at $t=0$, we obtain
  \begin{align*}
h_j(t)=\sum_{m=1}^\infty \alpha^j_m t^m, ~\alpha^j_m\in \mathbb C,~1\leq j\leq n.
\end{align*}
Hence, by (\ref{4-2018-20-1}) one must have  
   \begin{align} \label{4-2018-21-1}
a^1_1\alpha^1_1+\cdots+a^n_1\alpha^n_1&=0\nonumber;\\
a^1_1\alpha^1_2+\cdots+a^n_1\alpha^n_2+a^1_2 (\alpha^1_1)^2+\cdots+a^n_2(\alpha^n_1)^2&=0\nonumber;\\
\cdots\nonumber  \\
\sum_{m=1}^k a^1_m \sum_{\substack{n_1+\cdots+n_m=k\\ n_1,\ldots,n_m\geq 1}}\alpha^1_{n_1}\cdots\alpha^1_{n_m}+\cdots+ \\
+\sum_{m=1}^k a^n_m \sum_{\substack{n_1+\cdots+n_m=k\\ n_1,\ldots, n_m\geq 1}}\alpha^n_{n_1}\cdots\alpha^n_{n_m}&=0;\nonumber\\
\cdots \nonumber
\end{align}
  
Since $h_j, 1\leq j\leq n$, are all holomorphic on the unit disk $\Delta$, without loss of generality we can assume that $|\alpha^j_m|\leq 1$ for every $m\geq 1$ and $1\leq j\leq n$. Let us set $m^*=\min_{1\leq j\leq n}\{\nu(h_j)\}\geq 1$. Then, one has
 \begin{align}\label{22-10-1}
 \Big| \sum_{\substack{n_1+\cdots+n_m=k m^*\\ n_1,\ldots, n_m\geq m^*}}\alpha^j_{n_1}\cdots\alpha^j_{n_m}\Big|\leq (k m^*)^k\leq k^{2k}
\end{align}
  for every $k\geq m^*$ and $1\leq j\leq n$ and moreover,  there exists a $q\in \{1,2,\ldots,n\}$ such that $h_q\not \equiv 0$ and $0<|\alpha^q_{m^*}|\leq1$. Furthermore, by replacing $k$ by $k.m^*$, (\ref{4-2018-21-1}) yields that
 \begin{equation}\label{22-10-2}
 \begin{split}
\sum_{m=1}^k a^1_m \sum_{\substack{n_1+\cdots+n_m=km^*\\ n_1,\ldots,n_m\geq m^*}}\alpha^1_{n_1}\cdots\alpha^1_{n_m}+\cdots+ \\
+\sum_{m=1}^k a^n_m \sum_{\substack{n_1+\cdots+n_m=km^*\\ n_1,\ldots, n_m\geq m^*}}\alpha^n_{n_1}\cdots\alpha^n_{n_m}&=0, ~\forall\, k\geq 1.
\end{split}
\end{equation}
Therefore, by (\ref{22-10-1}) and (\ref{22-10-2}), one obtains
 \begin{align*}
 |\alpha^q_{m^*}|^k   |a^q_k|\leq k^{2k} \Big(\sum_{m=1}^{k-1} a^q_m+\sum_{j\ne q}\sum_{m=1}^k a^j_m\Big),~\forall\, k\geq m^*.
 \end{align*}
Now let us choose $k_0\in \mathbb N$ such that $k_0> \max\big\{\dfrac{1}{|\alpha^q_{m^*}|}, m^*\big\}$. Then, one gets
 \begin{align*}
|a^q_k|\leq  k^{3k}\Big(\sum_{m=1}^{k-1} a^q_m+\sum_{j\ne q}\sum_{m=1}^k a^j_m \Big), ~\forall\, k\geq k_0.
\end{align*}
This contradicts the condition (\ref{5-10-2019-1}).

  \noindent
 {\bf Case 3.}  $g\not \equiv 0, h_j\not\equiv 0$ for some $j\in\{1,2,\ldots,n\}$. Then, (\ref{4-2018-1}) becomes
 \begin{align*}
{\mbox{Re}}(g(t))+f_1(h_1(t))+\cdots+f_n(h_n(t)))=o(t^\infty).
\end{align*}
 Expanding $g, h_1,\ldots, h_n$ into the Taylor series at $t=0$, we have
  \begin{align*}
h_j(t)&=\sum_{m=1}^\infty \alpha^j_m t^m,~\alpha^j_m\in \mathbb C,~1\leq j\leq n;\\
g(t)&=\sum_{m=1}^\infty \gamma_m t^m,~\gamma_m\in \mathbb C.
\end{align*}
Hence, by (\ref{4-2018-1}) one must have  
  \begin{align} \label{4-2018-21-3}
\gamma_1+a^1_1\alpha^1_1+\cdots+a^n_1\alpha^n_1&=0\nonumber;\\
\gamma_2+a^1_1\alpha^1_2+\cdots+a^n_1\alpha^n_2+a^1_2 (\alpha^1_1)^2+\cdots+a^n_2(\alpha^n_1)^2&=0\nonumber;\\
\cdots\nonumber  \\
\gamma_k+\sum_{m=1}^k a^1_m \sum_{\substack{n_1+\cdots+n_m=k\\ n_1,\ldots, n_m\geq 1}}\alpha^1_{n_1}\cdots\alpha^1_{n_m}+\cdots+&\\
+\sum_{m=1}^k a^n_m \sum_{\substack{n_1+\cdots+n_m=k\\ n_1,\ldots, n_m\geq 1}}\alpha^n_{n_1}\cdots\alpha^n_{n_m}&=0;\nonumber\\
\cdots \nonumber
\end{align}

 Since $g, h_j, 1\leq j\leq n$, are all holomorphic on the unit disk $\Delta$, without loss of generality we can assume that $|\alpha^j_m|\leq 1,|\gamma_m|\leq 1$ for every $m\geq 1$ and $1\leq j\leq n$. As in Case $2$, we set $m^*=\min_{1\leq j\leq n}\{\nu(h_j)\}\geq 1$. Then, one also has
 \begin{align}\label{22-10-3}
 \Big | \sum_{\substack{n_1+\cdots+n_m=k m^*\\ n_1,\ldots, n_m\geq m^*}}\alpha^j_{n_1}\cdots\alpha^j_{n_m}\Big |\leq (k m^*)^k\leq k^{2k}
\end{align}
  for every $k\geq m^*$ and $1\leq j\leq n$ and there exists a $q\in \{1,2,\ldots,n\}$ such that $h_q\not \equiv 0$ and $0<|\alpha^q_{m^*}|\leq1$. Moreover, by replacing $k$ by $k.m^*$, (\ref{4-2018-21-3}) yields that
 \begin{equation}\label{22-10-4}
 \begin{split}
\gamma_{km^*}+\sum_{m=1}^k a^1_m \sum_{\substack{n_1+\cdots+n_m=km^*\\ n_1,\ldots,n_m\geq m^*}}\alpha^1_{n_1}\cdots\alpha^1_{n_m}+\cdots+ \\
+\sum_{m=1}^k a^n_m \sum_{\substack{n_1+\cdots+n_m=km^*\\ n_1,\ldots, n_m\geq m^*}}\alpha^n_{n_1}\cdots\alpha^n_{n_m}&=0, ~\forall\, k\geq 1.
\end{split}
\end{equation}
Therefore, by (\ref{22-10-3}) and (\ref{22-10-4}), we get
 \begin{align*}
 |\alpha^q_{m^*}|^k   |a^q_k|\leq k^{2k} \Big(\sum_{m=1}^{k-1} a^q_m+\sum_{j\ne q}\sum_{m=1}^k a^j_m\Big)+1,~\forall\, k\geq m^*.
 \end{align*}
Now we also choose $k_0\in \mathbb N$ such that $k_0> \max\big\{\dfrac{1}{|\alpha^q_{m^*}|}, m^*\big\}$. Then, one obtains 
\begin{align*}
|a^q_k|\leq k^k \Big( k^{2k}\big(\sum_{m=1}^{k-1} a^q_m+\sum_{j\ne q}\sum_{m=1}^k a^j_m\big)+1 \Big), ~\forall\, k\geq k_0.
\end{align*}
This again contradicts the condition (\ref{5-10-2019-1}).  
  
  Altogether, the proof is complete.
  \end{proof}
  
\begin{proof}[Proof of Corollary \ref{Cor1}]     
  Let $M$ be the smooth pseudoconvex real hypersurface given in the proof of Theorem \ref{T4}. Then $M$ is of Bloom-Graham infinite type at $p=0$ and moreover it does not admit any nonconstant holomorphic curve tangent to $M$ at $p=0$ to infinite order. We shall show that $M$ does also not admit any complex submanifold $X_\infty$ of codimension one in $\mathbb C^{n+1}$ tangent to $M$ at $0$ to infinite order. That is, $\tilde \rho\mid_{X_\infty}$ vanishes to infinite order at $p=0$.

Indeed, suppose otherwise. Then, $\tilde \rho\mid_{X_\infty\cap\{z_2=\cdots=z_n=0\}}$ vanishes to infinite order at $p=0$. Note that $X_\infty\cap\{z_2=\cdots=z_n=0\}$ is locally represented as the graph of a holomorphic curve. Therefore, there exists a nonconstant holomorphic curve $\gamma_\infty: (\mathbb C,0)\to (\mathbb C^{n+1},0)$ tangent to $M$ at $0$ to infinite order, which is a contradiction. Hence, this completes the proof. 
\end{proof}
In order to give a proof of Theorem \ref{T1}, we need the following lemma.
 \begin{lemma}\label{BL-1} Let $n\geq 1$. If $(M, p)$ is a smooth real hypersurface germ of Bloom-Graham infinite type in $\mathbb C^{n+1}$, then $(M, p)$ is of D'Angelo infinite type.
 \end{lemma}
 \begin{proof}
 Suppose that $(M, p)$ is a smooth real hypersurface germ of Bloom-Graham infinite type in $\mathbb C^{n+1}$ and let $r$ be a local defining function for $M$ near $p$. By definition, it follows that there exists a sequence $\{X_m\}_{m=1}^\infty$ of codimension-one complex submanifolds of $\mathbb C^{n+1}$ such that $p\in X_m$ and
 $$
 ord_p(r\mid_{X_m})\geq m
 $$
 for every $m=1,2,\ldots$.
 
 For each $m=1,2,\ldots$, choose a regular holomorphic curve $\gamma_m:\Delta:=\{z\in \mathbb C\colon |z|<1\}\to X_m \subset \mathbb C^{n+1}$ such that $p=\gamma_m(0)$ and $\nu_0(\gamma_m)=1$. Then, $\nu_0(r\circ \gamma_m)\geq ord_p(r\mid_{X_m})\geq m$. This yields that the D'Angelo type of $M$ at $p$ equals $+\infty$. Therefore, the proof is complete.
 \end{proof}

\begin{proof}[Proof of Theorem \ref{T1}]  
Let $M$ be the smooth pseudoconvex real hypersurface given in the proof of Theorem \ref{T4}. Then $M$ is of Bloom-Graham infinite type at $p=0$. By Lemma \ref{BL-1}, $M$ is of D'Angelo infinite type at $p=0$. Hence, the proof follows from Theorem \ref{T4}.
\end{proof}
\begin{remark}\label{R1} Let $X$ be a real variety in $\mathbb C^n$ defined by $f_1(z_1)+\cdots+f_n(z_n)=0$, where $f_j, 1\leq j\leq n$, given in the proof of Theorem \ref{T4}. Following the proof of Theorem \ref{T4}, we conclude that $X$ does not admit any nonconstant holomorphic curve $\gamma: (\mathbb C, 0) \to (\mathbb C^n, 0)$ with convergent Taylor series tangent to $X$ at $0$ to infinite order. Moreover, let $\widetilde X$ be a formal real variety in $\mathbb C^n$ defined by $\tilde f_1(z_1)+\cdots+\tilde f_n(z_n)=0$, where $\tilde f_j, 1\leq j\leq n$, respectively, are the (divergent) Taylor series at the origin of smooth harmonic functions $ f_j, 1\leq j\leq n$, given in the proof of Theorem \ref{T4}. Following the proof of Theorem \ref{T4}, we also conclude that $\widetilde X$ does not  contain any nonconstant holomorphic curve $\gamma: (\mathbb C, 0) \to (\mathbb C^n,0)$ with convergent Taylor series.
\end{remark}

\begin{Acknowledgment}  Part of this work was done while the second author was visiting the Vietnam Institute for Advanced Study in Mathematics (VIASM). He would like to thank the VIASM for financial support and hospitality. We gratefully acknowledge the careful reading by the referees.
\end{Acknowledgment}

\end{document}